\documentclass[12pt,leqno]{amsart}
\usepackage{amsfonts,amsmath,amssymb,amsthm, geometry}
\usepackage{graphicx} 
\usepackage{srcltx}
\usepackage[all]{xy}

\newtheorem{theorem}{Theorem}[section]
\newtheorem{corollary}[theorem]{Corollary}
\newtheorem{lemma}[theorem]{Lemma}
\newtheorem{proposition}[theorem]{Proposition}
\theoremstyle{definition}
\newtheorem{definition}[theorem]{Definition}
\theoremstyle{remark}
\newtheorem{remark}[theorem]{\sc Remark}
\newtheorem{example}[theorem]{\sc Example}

\renewcommand{\Box}{\square}    

\newcommand{\diffeo}{{\rm{diffeo}}}

\newcommand{\Sing}{{\rm{Sing\hspace{2pt}}}}

\newcommand{\im}{{\rm{Im\hspace{2pt}}}}

\newcommand{\ity}{{\infty}}

\newcommand{\m}{\setminus}
\newcommand{\Si}{S_\infty}

\newcommand{\fin}{\hspace*{\fill}$\Box$\vspace*{2mm}}

\newcommand{\cL}{{\mathcal L}}

\newcommand{\bR}{{\mathbb R}}
\newcommand{\bC}{{\mathbb C}}

\newcommand{\bN}{{\mathbb N}}

\begin{document}

\title[Families of real curves with more than one parameters]{Bifurcation values of families of real curves}

%

\author{\sc Cezar Joi\c ta}
\address{Laboratoire Europ\' een Associ\'e  CNRS Franco-Roumain Math-Mode,
Institute of Mathematics of the Romanian Academy, P.O. Box 1-764,
 014700 Bucure\c sti, Romania}

\email{Cezar.Joita@imar.ro}

\author{Mihai Tib\u ar}
\address{Universit\'e de Lille 1, CNRS, UMR 8524 - Laboratoire Paul Painlev\'e, F-59000 Lille, France}
\email{tibar@math.univ-lille1.fr}




\thanks{The authors acknowledge the support of the Labex CEMPI
(ANR-11-LABX-0007-01). C. Joi\c ta acknowledges the
CNCS grant PN-II-ID-PCE-2011-3-0269.}


\begin{abstract}
In more than two variables, detection of the bifurcation set of polynomial mapping $\bR^n \to \bR^p$,  $n\ge p$, 
is a still  unsolved problem. In this note we provide a solution for $n= p+1 \ge 3$.
\end{abstract}

\maketitle


\section{Introduction}\label{s:intro}

 The bifurcation locus of a polynomial mapping $F : \bR^n \to \bR^p$, $n\ge p$,  is the minimal set of points $B(F) \subset \bR^p$ outside which the mapping is a C$^\ity$ locally trivial fibration. Unlike the local setting, the critical locus $\Sing F$ is not the only obstruction to the existence of  fibrations in the global setting. The simplest evidence of such a phenomenon in case $p=1$ is in the example of $f(x,y) = x + x^2y$, where $\Sing f =\emptyset$ but $B(F) = \{ 0\}$.
 In case $p>1$, Pinchuk \cite{Pi}  provided an example of a polynomial mapping $F: \bR^2 \to \bR^2$ where  $\Sing F =\emptyset$ but $B(F) \not= \emptyset$, which is a negative answer to the Jacobian Conjecture over the reals.  

In more than two variables, over the last 20 years one could only  estimate $B(F)$ by supersets $A \supset B(F)$ according to certain \emph{regularity conditions at infinity} \cite{Ti-reg}, \cite{Ra}, \cite{KOS}, \cite{Ti},
\cite{CT}, \cite{DRT} etc.  The bifurcation set $B(F)$  was shown to be detectable precisely only if $p=1$ and $n=2$, see  \cite{TZ},  \cite{CP}, \cite{HN}. A similar situation holds over the complex field, with a large number of articles in the last decades (see e.g. \cite{Ti} for references before 2007).

We address here the problem of detecting the bifurcation set in algebraic families of real curves of more than one parameter, in particular the case $n = p+1 \ge 3$. The methods developed in \cite{CP} or \cite{HN} cannot be extended beyond two variables since they are based essentially on the use of the ``polar locus'' or the ``Milnor set'' (see Definition \ref{d:rho-reg}) which are of dimension 1 only in the $n=2$ case. 
 Our task was to find a way to extend to higher dimensions the ideas established in \cite{TZ} for $n=2$. As a matter of fact we have to change the viewpoint of  \cite{TZ} and find completely new definitions for the \emph{non-vanishing} condition  and for the \emph{non-splitting} condition. 
We then get the following extension of the main result \cite{TZ}, keeping its  spirit and terminology.

\begin{theorem} \label{t:main}
 Let $X\subset \bR^m$ be a real nonsingular irreducible algebraic set of dimension $n\ge 3$ and  let  $F : X\to \bR^{n-1}$ be an algebraic map. Let $a$ be an interior point of the set $\im F \setminus \overline{F(\Sing F)}\subset \bR^{n-1}$ and let $X_{t} := F^{-1}(t)$.  Then $a\not\in B(F)$ 
if and only if the following two conditions are satisfied:
\begin{enumerate}
 \item 
 the Euler characteristic $\chi(X_{t})$ is constant when 
$t$ varies within some neighbourhood of $a$, 
and  

\item there is no  
component of $X_t$ which vanishes at infinity as $t$ tends to $a$.
\end{enumerate}
\smallskip
\noindent
The above criterion (a)+(b) may be replaced by (a')+(b') where:
\begin{enumerate}
 \item[(a')] the Betti numbers of $X_{b}$ are constant for $b$ in some neighbourhood of $a$,
and   
\item[(b')] there is no splitting at infinity at $a$.
\end{enumerate}
\end{theorem}

Let us point out that the Euler characteristic of regular fibres is given by the following simple formula:
\[ \chi(X_{t}) = \frac{1}{2}  \lim_{R\to \ity} \# [X_t \cap S_R]
\]
 where $S_R \subset \bR^m$ denotes the sphere of radius $R$ centred at the origin.

In order to situate our study in the mathematical landscape, we start with  discussing in \S \ref{s:pinchuk} the real counterpart of several results well-known in the complex setting.


\section{Real versus complex setting} \label{s:pinchuk}

\subsection{The Abhyankar-Moh-Suzuki theorem}
  The famous example by Pinchuk \cite{Pi}  yields a polynomial mapping $\bR^2 \to \bR^2$ with no singularities but which is not a global diffeomorphism, thus providing a counter-example to the strong Jacobian Conjecture over the reals. The Jacobian problem remains nevertheless open over $\bC$. 

 We may then further ask what happens when a polynomial map is a component of a global diffeomorphism since,  over the complex field, one has the following well-known Abhyankar-Moh-Suzuki theorem  \cite{AM}, \cite{Su}:
 \emph{ A complex polynomial function $f : \bC^2 \to \bC$ which is a locally trivial fibration is actually equivalent to a linear function, modulo automorphisms of $\bC^2$. }

This result is again not true over $\bR$ and it is actually not difficult to find examples like the following:
\begin{example}\label{e:AbhMohSuz}
The polynomial function
 $g : \bR^2 \to \bR$, $g(x,y) = y (x^2 + 1)$ is a component of a diffeomorphism, fact that one can see 
by using the change of variables $(x,y) \mapsto (x, \frac{y}{x^2 +1})$. Therefore $g$ is a globally trivial fibration. However, $g$ cannot be linearised by a \emph{polynomial} automorphism. 

\end{example}


\subsection{The Euler characteristic test}\label{ss:euler}
The following result was found in the 70's \cite{Su}, \cite{HL}: 
\emph{Let $f : \bC^2 \to \bC$ be a polynomial function and let $a\in \bC\m f(\Sing f)$. Then $a \not\in B(f)$ if and only if the Euler characteristic of the fibres $\chi(f^{-1}(t))$ is constant for $t$ varying in some neighbourhood of $a$.}

Its real counterpart came out much later. It appears that for polynomial functions $\bR^2 \to \bR$
the constancy of the Euler characteristic of the fibres is not sufficient and that other phenomena may occur at infinity: the ``splitting'' or the ``vanishing'' of components of fibres (see Definition \ref{d:vanish}).

\begin{theorem}\label{t:2var} \rm  \cite{TZ}
\label{t:intro} \it 
Let $X$ be a real algebraic nonsingular surface and let  $\tau : X\to \bR$ be an algebraic map. Let $a\in \im \tau$ be a regular value of  $\tau$, and let $X_t := F^{-1}(t)$.  Then $a\not\in B(\tau)$ 
if and only if:
\begin{enumerate}
 \item 
 the Euler characteristic $\chi(X_{t})$ is constant when 
$t$ varies within some neighbourhood of $a$, and  

\item there is no  
component of $X_t$ which vanishes at infinity as $t$ tends to $a$. \fin
\end{enumerate}
\end{theorem} 

One moreover shows that the above criterion (a)+(b) is equivalent to the 
following:\\
\ \ \ (c)\emph{ the Betti numbers of $X_{t}$ are constant for $t$ in some neighborhood of $a$},
and \\ 
\ \ \ (d)\emph{ there is no  
component of $X_t$ which splits at infinity as $t$ tends to $a$.}
\\

 All the above conditions (a)--(d) are necessary but none of them implies alone the local triviality of the map $\tau$, as the examples in \cite{TZ} show.
Our Theorem \ref{t:main} represents the extension of the above result to algebraic families of curves of more than one parameter.

\subsection{Detecting bifurcation values by the Milnor set} 

It was shown in \cite{Ti-reg}, \cite{DRT} that, in case of a polynomial map $F : \bR^n \to \bR^p$, the bifurcation non-critical locus  $B(F) \m f(\Sing f)$  is included in the set of ``$\rho$-nonregular values at infinity''. The\emph{ $\rho$-regularity} is a ``Milnor type'' condition that controls the transversality of the fibres of $F$ to the spheres centered at $c\in\bR^n$, more precisely:

\begin{definition}\label{d:milnor}\label{d:rho-reg}
Let $F\colon\bR^n\to\bR^p$ be a polynomial map,  where $n\geq p$. Let  $\rho_c\colon\bR^n \to \bR_{\geq 0}$ be the Euclidian distance function to the point $c\in\bR^n$.  We call \emph{Milnor set of $(F,\rho_c)$} the critical set of the mapping $(F, \rho_c)\colon\bR^n\to\bR^{p+1}$ and  denote it by $M_c(F)$. 
  We call:
\[
 S_c(F):=\{t_0\in\bR^p\mid\exists \{ x_j\}_{j\in \bN}\subset M_c(F), \lim_{j\to\infty}\| x_j\|=\infty\mbox{ and }\lim_{j\to\infty}F(x_j)=t_0\}
\]
the set of \emph{$\rho_c$-nonregular values at infinity}. If $t_0\notin S_c(F)$ we say that $t_0$ is \textit{$\rho_c$-regular at infinity}. We set  
$\Si (F):= \bigcap_{c\in \bR^n} S_c(F)$.
\end{definition}

In case of polynomials $f : \bC^2 \to \bC$ the following characterisation has been proved \cite[Cor.5.8]{ST}, \cite[Thm.2.2.5]{Ti}:  \emph{Let $a\in \bC\m f(\Sing f)$. Then $a\in B(f)$ if and only if $a \in S_0(f)$}.
 
This is not true anymore over the reals, as shown by the following example from \cite{TZ}: 
 $f:\mathbb{R}^{2}\rightarrow\mathbb{R}$, $f(x,y)=y(2x^{2}y^{2}-9xy+12)$, where $S_0(f)$ contains the origin of $\bR$ but the bifurcation set $B(f)$ is empty.

However, with some more information along the branches of the Milnor set $M_c(f)$ which take into account the ``vanishing'' and the ``splitting'' phenomena at infinity (see Definitions \ref{d:vanish} and \ref{d:split}), one is able to 
produce a criterion, as follows. First, there is some open dense set $\Omega_f \subset \bR^2$ such that for $c\in \Omega_f$ the Milnor set $M_c(f)$ is a curve (or it is empty). For such a point $c\in \Omega_f$  one counts the number $\# [X_t^j \cap M_c(f)]$ of points of intersection of the connected components $X_t^j$ of the fibres $X_t$ with the curve $M_c(f)$. The following criterion holds:  \emph{Let $a\in \bR\m f(\Sing f)$. Then  $a\in B(f)$ if and only if $a \in S_c(f)$ and $\lim_{t\to a}\# [X_t^j \cap M_c(f)] \not\equiv 0\ (mod\ 2)$ for some sequence of connected components $X_t^j$ of $X_t$.}
 This can be easily proved by using the results of our paper and is close to the main theorem of \cite{HN} which is proved for the larger class of polynomial functions defined on a smooth non-compact affine algebraic surface $X$. One of the significant  difference between our approach and that of \cite{HN} is that we test 
 \emph{connected components} $X_t^j$ of fibres and not just the fibres of $f$ as in \emph{loc.cit.} The reason is that one may have vanishing and splitting at infinity in two different components of the same fibre, with one maximum and one minimum which would cancel in the framework  of  \cite{HN} but not in the above statement.\footnote{see also \cite[\S 3 and Ex. 3.1]{TZ} for the construction of such examples.}





\section{The non-vanishing condition}\label{s:def}

\subsection{Non-vanishing at infinity}\label{ss:mu}
Let $X\subset \bR^m$ be a real nonsingular irreducible algebraic set of dimension $n$,   and let   $F : X\to \bR^{n-1}$ be an algebraic map. 
Throughout this section the point $a$ will denote an interior point of $\im F \setminus \overline{F(\Sing F)}$. 

 As before we denote by $X_b$ the fibre $F^{-1}(b)$. 
Let then $X_b = \sqcup_j X_b^j$ be the decomposition of the fibre $X_b$ into connected components.
Define:

\[ \mu(b)  := \max_j \inf_{x\in X_b^j} \|x\| \]

\begin{definition}\label{d:vanish}
We say that there is \emph{vanishing at infinity at $a\in\bR^{n-1}$}
if there exists a sequence of points $a_k \to a$ such that $\lim_{k\to \infty}  \mu(a_k) = \infty$.

If there is no such sequence,  we say that \emph{there is no vanishing at $a\in\bR^{n-1}$}  and we denote this situation shortly by $NV(a)$.
\end{definition}

\begin{remark}\label{r:nv}
One can easily deduce from the above definition that $NV$ is an open condition.
\end{remark}


\subsection{Proof of Theorem \ref{t:main}, first part}

The regular fibres of $F$ are 1-dimensional manifolds, hence every such fibre is a
 finite union of connected components. Each such component is either compact and thus diffeomorphic to a circle, or non-compact and thus diffeomorphic to the affine line $\bR$.
Let us denote by $s(b)$ the number of compact components of the fibre $F^{-1}(b)$
and by $l(b)$ the number of non-compact components of this fibre. Let us note that these definitions make sense 
for a semi-algebraic set $X$; we shall occasionally use them in such a context in the proofs below.
 
Let $a\in \bR^{n-1}$ be as in the statement of Theorem \ref{t:main} and let us assume $NV(a)$. 
By Remark \ref{r:nv}, there exists  a ball $D$ centered at $a$, included in the interior of the set $\im F \setminus \overline{F(\Sing F)}\subset \bR^{n-1}$ such that $NV(b)$ for any $b\in D$. For such a ball $D$,  we show:

\begin{lemma}\label{l:numbers}
 The numbers $s_X(b)$ and $l_X(b)$ are constant for $b\in D$.
\end{lemma}
\begin{proof}
  Let us fix some point $b\in D$ and let $L_{ab}\subset \bR^{n-1}$ denote the unique line passing through the points $a$ and $b$. The fibre $X_t$ is a 1-dimensional manifold for any $t\in D$, in particular the inverse image $F^{-1}(L_{ab})$ is an algebraic family of non-singular real curves. It is known (as proved by Thom, Verdier and others, see e.g. \cite[Cor. 1.2.13]{Ti}) that the projection $\tau_{ab} : F^{-1}(L_{ab}) \to L_{ab}$  has a finite number of atypical values. In the hypotheses of Theorem \ref{t:main} and by Remark \ref{r:nv}, at each supposed atypical value of $L_{ab}\cap D$ one may apply Theorem \ref{t:2var} for $\tau_{ab}$. This leads to the conclusion that there are no atypical values of $\tau_{ab}$ on $L_{ab}\cap D$, in particular 
the restriction of $F$ is a locally trivial fibration over $L_{ab}\cap D$, hence a trivial fibration.  This implies $s_X(b) = s_X(a)$ and $l_X(b)= l_X(a)$.
\end{proof}

\subsection{Compact components}\label{ss:compact}
Let us consider some compact connected component of the regular fibre $X_a$, if there is such.
Then this compact component may be covered by finitely many open connected sets $B_i\subset X$ such that $B_i\cap X_a$ is connected and that the restriction $F_| : B_i \to F(B_i)$ is a trivial fibration. In particular each fibre of this fibration is connected.  There exists a small enough closed  ball $D \subset \bR^{n-1}$ centered at $a$ which is contained in all images $F(B_i)$. It then follows that the restriction
$F_| :  F^{-1}(D) \cap \cup_{i} B_i \to D$
is a proper submersion. Therefore, by Ehresmann's fibration theorem, this is a locally trivial fibration, hence trivial (since $D$ is contractible).

It follows that, for any $t\in \mathring{D}$, there is a unique connected component of the fibre $X_t$ which intersects the open set $F^{-1}(\mathring{D})\cap \cup_{i} B_i$.

It also follows  that $\mathcal{D}  := F^{-1}(\mathring{D})\cap \cup_{i} B_i$ is an open connected component of  $F^{-1}(\mathring{D})$. Therefore 
 $F^{-1}(\mathring{D}) \m \mathcal{D}$ is an open subset of $F^{-1}(\mathring{D})$.

By Lemma \ref{l:numbers}  and by taking an eventually smaller ball $D$, we have that  for any $t\in \mathring{D}$, $X_t \cap F^{-1}(\mathring{D}) \m \mathcal{D}$ has precisely $l_X(a)$ connected non-compact components and  $s_X(a) -1$  connected compact components.

In this way we have produced a trivialisation on a connected component of $F^{-1}(\mathring{D})$
and we have reduced the problem to  constructing a trivialisation 
within the set $F^{-1}(\mathring{D}) \m \mathcal{D}$, where the numbers are: 
\[s_{F^{-1}(\mathring{D}) \m \mathcal{D}}(a) =  s_X(a) -1 \  \mbox{\  and \ } l_{F^{-1}(\mathring{D}) \m \mathcal{D}}(a) =  l_X(a). \]

We apply the above procedure until  we eliminate one by one all the compact components.  We may then assume from now on that the fibre $X_t$ has no compact component, for any $t$ in some neighbourhood of $a$.

\subsection{Line components}\label{ss:line}

Consider a line component $X_a^1$ of $X_a$ and fix some point $p\in X_a^1$. 
Since $F$ is a submersion at $p$, there exists a small ball $B_p$ at $p$ such that $B_p \cap X_a$ is connected and that 
the restriction of $F$ to $B_p\cap F^{-1}(D)$ is is a trivial fibration over a small enough disk $D\subset F(B_p)$ centered at $p$. 
It follows that, for any $t\in D$, the intersection $X_t\cap B_p$ is  connected and thus included into a unique connected component of the fibre $X_t$.
 
Let $\cL_1$ denote the union over all $t\in D$ of the connected components of the fibres $X_t$
which intersect $B_p$. 
Note that each such connected component is a line component, since we have assumed that $s_X(a) = 0$, thus $s_X(t) = 0$ for all $t\in D$ (by reducing the radius of $D$, if needed), by Lemma \ref{l:numbers}.

We have thus associated the connected set $\cL_1$ to the chosen component  $X_a^1$.  Consider
the similar construction for each other connected component of $X_a$. Namely we start like above by choosing one point $p_i$ on each component of $X_a$ and some ball $B_{p_i}$ at $p_i$. We get in this way 
the sets $\cL_1, \cL_2, \cdots , \cL_{l_X(a)}$ where we recall that $l_X(a)$ denotes the number of connected components of $X_a$ and that this number is a local invariant over the target set, by Lemma \ref{l:numbers}. Without lowering the generality,
we may assume that the ball $D$ in the target is common to all these constructions. 

It then follows that the sets $\cL_i$ are all \textit{connected} (by definition) and \textit{pairwise disjoint}. Indeed, if this is not true, then there is 
some $t\in D$ such that the fibre $X_t$ has a connected component which belongs to more than one set $\cL_i$. 
 But by the above construction each $\cL_i$ contains precisely one connected component of $X_t$ and the number of connected components of $X_t$ is precisely $l_X(a)$ by Lemma \ref{l:numbers}.  We thus obtain a numerical contradiction.

Let us show that the sets $\cL_i$ are also \textit{open} and therefore they are manifolds. 
Let us fix $i$ and fix some $q\in X_b \cap \cL_i$ for some $b\in D$ as above. 
There exists a ball $B_q$ which has the properties of the ball $B_{p_i}$ considered above. This implies that a unique component of each fibre $X_t$ intersects $B_q$, for $t$ in  some small enough ball $D'\subset D$ centered at $b$. We claim that the component of $X_t$ intersecting $B_q$ is precisely the component belonging to $\cL_i$, as follows. Let $q_i \in X_b\cap B_{p_i}$. We consider a non self-intersecting analytic path in $X_b$ starting at $q_i$ and ending at $q$. Since compact, this can be covered by finitely many 
small balls $B_j$ with the same properties of $B_q$ or $B_{p_i}$. We then apply the reasoning of \S \ref{ss:compact} above to get that the restriction
$F_| :  F^{-1}(D) \cap \cup_{j} B_j \to D'$, for some small enough $D'$, 
is a proper submersion. Therefore, by Ehresmann's fibration theorem, this is a locally trivial fibration, hence trivial, since $D'$ is contractible. Since the fibres of this map are connected by our construction and since each of them intersects $B_{p_i}$, it follows that each fibre of $F_|$ is included into the corresponding fibre of $\cL_i$. Since $F^{-1}(D') \cap \cup_{j} B_j$ is in particular a neighbourhood of the point $q\in \cL_i$, this finishes the proof of our claim.

We conclude that the open sets $\cL_i$ together provide a partition of $F^{-1}(D)$ into open manifolds.
We  may then apply \cite[Proposition 2.7]{TZ} stated below in order to conclude that every restriction
$F_| : \cL_i \to D$ is a trivial fibration. This ends the proof of the first part of our theorem.
\fin

\begin{proposition}\label{p:bundle} \cite[Prop. 2.7]{TZ} 

Let $M \subseteq \bR ^n$ be a smooth submanifold of 
dimension $m+1$ and let $g : M \to \bR^m$ be a smooth function without singularities and such that all its fibres $g^{-1}(t)$ are closed in $\bR^n$ and diffeomorphic to $\bR$. 
Then $g$ is a {\rm C}$^{\infty}$ trivial fibration. 
In particular, $M\stackrel{\diffeo}{\simeq} \bR^{m+1}$.
\fin
\end{proposition}

\begin{remark}
 It is interesting to point out that the sets $\cL_i$ may be defined without the non-vanishing condition at $a$,  but then the sets $\cL_i$ may not exhaust $F^{-1}(D)$ or they may be not mutually disjoint. The first phenomenon is due to the vanishing of components and the second is due to the so-called ``splitting'' phenomenon which we present in the next section. 
\end{remark}

\section{The non-splitting condition}\label{s:ns}

We study here the phenomenon of splitting at infinity in families of curves of several parameters. The following definition of limit sets was used in a particular setting in \cite{TZ} and corresponds to the notation ``limsup'' used in \cite{DD}. We have learned from \cite{DD} that such limits have been considered classically by Painlev\' e and Kuratowski. 

\begin{definition}
Let $\{M_k\}_k$ be a sequence of subsets of $\bR^{m}$. A point $x\in\bR^{m}$ is called \emph{limit point}
of  $\{M_k\}_k$ if there exists a sequence of points $\{x_i\}_{i\in \bN}$ with $\lim_{i\to\infty}x_i=x$ and such that $x_i\in M_{k_i}$ for some integer sequence $\{k_i\}_i \subset \bN$ with $\lim_{i\to\infty}k_i=\infty$.

 The set of all limit points of $\{M_k\}_k$ will be denoted by $\lim M_k$. 
\end{definition}

In the remainder of this paper the point $a$ will be an  interior point of $\im F \setminus \overline{F(\Sing F)}\subset \bR^{n-1}$, like in the statement of Theorem \ref{t:main}.

\begin{remark}\label{r:2.3}
Let $\{b_k\}_{k\in \bN}$ be a sequence of points in $\im F \setminus \overline{F(\Sing F)}$ such that 
$b_k\to a$ and that, for each $k$, $X^j_{b_k}$ is a fixed connected component of 
$X_{b_k}$. Then  $\lim X^j_{b_k}$ is either empty or a union of connected components of $X_a$. This is a more precise version of \cite[Lemma 2.3(i)]{TZ} and follows from the definition of the limit and from the fact that $a$ is a regular value of $F$.
\end{remark}

\begin{definition}\label{d:split}
We say that there is \emph{no splitting at infinity at $a\in\bR^{n-1}$}, and we abbreviate this by $NS(a)$, if 
the following holds: let $\{b_k\}_{k\in \bN}$ be a sequence in $\bR^{n-1}$ such that $b_k\to a$ and let $\{p_k\}_{k\in \bN}$
be a convergent sequence in $X$ such that $F(p_k)=b_k$.  If $X_{b_k}^j$ denotes the connected component
of $X_{b_k}$ which contains $p_k$, then  the limit set $\lim X^j_{b_k}$ is connected.

 We say that there is \emph{strong non-splitting at infinity at $a\in\bR^{n-1}$}, and we abbreviate this by $SNS(a)$, if in addition to the definition of $NS(a)$ we ask the following: if all the components $X_{b_k}^j$ are compact then  the limit $\lim X^j_{b_k}$ is compact too.
\end{definition}

This notion of ``non-splitting'' $NS$ extends the one introduced in \cite{TZ} for $n=2$.

\begin{remark}\label{rmns}\ 
\begin{enumerate}
\item
 For two sequences $\{b_k\}_{k\in \bN}$ and $\{p_k\}_{k\in \bN}$ as above, if we denote by $X^j_a$ the connected component  of $X_a$ which contains
$p:=\lim p_k$ and by $X^j_{b_k}$ the connected component of $X_{b_k}$ which contains $p_k$ then, by Remark
\ref{r:2.3}, we have the inclusion $X^j_a\subset \lim X^j_{b_k}$. Therefore $NS(a)$ means that $\lim X^j_{b_k} = X^j_a$.

\item We do not know whether $NS(a)$ implies $NS(b)$ for $b$ in a small enough neighborhood of $a$. However this is true
whenever the  Betti numbers of $X_{b}$ are constant for $b$ in a neighbourhood of $a$. This follows from the second part of the proof
of Theorem \ref{t:main} presented bellow.
\end{enumerate}
\end{remark}

\subsection{Proof of Theorem \ref{t:main}, second part}

Conditions (a') and (b') are obviously necessary for $a\not\in B(F)$. Let us prove that they imply the conditions (a) and (b) of Theorem  \ref{t:main}. Since condition (a)  is obviously implied by condition (a'), the rest of the proof will be devoted to show condition (b).

 Let us denote by $X_a^1,\ldots, X_a^{l}$ the connected components of $X_a$.
  For each $j=1,\ldots,l$, we choose a point  $z_j\in X_a^j$ and, like in \S \ref{ss:compact},
we fix a small enough ball $B_j$ at $z_j$ such that $B_j \cap X_a$ is connected and that 
the restriction of $F$ to $B_j\cap F^{-1}(D_j)$ is a trivial fibration over a small enough disk $D_j\subset F(B_j)$ centered at $a$.
We may assume that the small balls $B_1,\ldots, B_l$ are pairwise disjoint. In particular for each $b\in \cap_j D_j$ we have that $B_j$ intersects exactly one connected component of $X_b$. 
We therefore may define a function $\Phi_b$ on the set $\{1,\ldots,l\}$ with values in the set of connected components $X_b^1,\ldots, X_b^{s_b}$ of $X_b$ by setting
$\Phi_b(j)$ to be the unique component of $X_b$ which intersects $B_j$.

\smallskip
\noindent
{\it Claim:} $NS(a)$ implies that there exists  a ball  $D\subset \cap_j D_j$ centered at $a$ such that, for any $b\in D$, $\Phi_b$ is a bijection.

\smallskip
\noindent
{\it Proof of the claim.}
Since $b_0(X_t)$ is constant at $a$, there is a small enough disk $D'$ centered at $a$ (which we may assume included in $\cap_j D_j$) such that $s_b =l$, for all  $b\in D'$. It is therefore enough to prove that $\Phi_b$ is injective on some small enough disk $D\subset D'$ centered at $a$.  By reductio ad absurdum, suppose that there exists a sequence of 
points $\{b_k\}_{k\in \bN}$ in $\bR^{n-1}$
such that $b_k\to a$ and $i_k,j_k\in\{1,\ldots, l\}$, $i_k\neq j_k$, such that $\Phi_{b_k}(i_k)=\Phi_{b_k}(j_k)$. Since the set of all subsets
with exactly two elements of $\{1,\ldots, l\}$ is finite, by passing to a subsequence, we may assume that
there exist $i,j\in \{1,2\ldots, l\}$, $i\neq j$, such that  $\Phi_{b_k}(i)=\Phi_{b_k}(j)$ for every $k$.  We get that the limits $\lim \Phi_{b_k}(i)$ and $\lim \Phi_{b_k}(j)$ coincide and, by Remark \ref{rmns}(a), that they are equal to some connected component of $X_a$.

On the other hand, since 
$F_{|B_i\cap F^{-1}(D_i)}$ and $F_{|B_j\cap F^{-1}(D_j)}$ are trivial fibrations it follows that the sets
 $B_i\cap F^{-1}(D_i)\cap\lim_k \Phi_{b_k}(i)$
and $B_j\cap F^{-1}(D_j)\cap\lim_k \Phi_{b_k}(j)$ are non-empty and they are contained in different components of $X_a$. This yields a contradiction.  Our claim is proved.
\medskip

Finally, let us show that we have $NV(a)$. If this were not the case then there exists a sequence
$\{b_k\}_{k\in \bN}$ converging to $a$ such that  $\lim_{k\to \infty}  \mu(b_k) = \infty$  (cf Definition of $\mu$ in \S \ref{ss:mu}). This implies that there is a connected
component $X^j_{b_k}$ of  $X_{b_k}$ such that $X^j_{b_k}\cap(\cup_1^l B_j)=\emptyset$ and this contradicts the surjectivity of
$\Phi_{b_k}$.

This ends the proof of the reduction of the second part of Theorem \ref{t:main} to its  first part.\fin

\begin{remark}
 In the above proof we need to assume the constancy of the Betti number $b_1(X_t)$ since this 
condition is \emph{not implied} by the constancy of the Betti number $b_0(X_t)$, by $NS(a)$ and by $NV(a)$ together. The reason of this behavior, which can be seen in \cite[Example 3.2]{TZ},  is the phenomenon of ``breaking'' of oval components at infinity. Nevertheless such loss of points at infinity can be avoided if instead of $NS(a)$ we ask the $SNS(a)$ condition of Definition \ref{d:split}, as shown by the following result.
\end{remark}

\begin{corollary}
In the conditions of Theorem \ref{t:main}, the following equivalence holds:\\
\centerline{  $a\not\in B(F)$ $\Leftrightarrow$ $SNS(a)$ and $NV(a)$.}
\end{corollary}
\begin{proof}
Conditions $SNS(a)$ and $NV(a)$ are obviously necessary for $a\not\in B(F)$. Let us show the sufficiency. 
By Remark  \ref{rmns}(a), $NS(a)$ implies the inequality $b_0(X_t) \ge b_0(X_a)$ for $t$ in some small enough disk centered at $a$.
Next, $NS(a)$ together with $NV(a)$ imply that this inequality is an equality. What we only need in order to conclude is the constancy of $b_1(X_t)$ for $t$ in some neighbourhood of $a$, but this is exactly what the condition $SNS(a)$ insures.
\end{proof}

The conditions  $NV(a)$,  $NS(a)$ (hence $SNS(a)$ too) are conditions ``at infinity'', more precisely one can prove the following statement  in a similar way as above.

\begin{theorem} \label{t:main2}
 Let $X\subset \bR^m$ be a real nonsingular irreducible algebraic set of dimension $n$ and  let  $F : X\to \bR^{n-1}$ be an algebraic map. Let $a\in \im F$ be a regular value of $F$ and let $R\gg 1$ be large enough such that 
$X_a$ is transversal to the sphere $X\cap S_R^{m-1}$. Let us denote by $G$ the restriction of $F$ to $X\m B_R^{m}$ and by $X_t$ its fibres.

 If  $a$ is an interior point of the set $\im G \setminus \overline{G(\Sing G)}\subset \bR^{n-1}$, then $a\not\in B(G)$  if and only if we have either conditions (a) + (b) or conditions (a') + (b') of Theorem \ref{t:main}.\fin
\end{theorem}



\begin{thebibliography}{MMMM}

\bibitem[AM]{AM}
S.S. Abhyankar, T.T. Moh, \emph{Embeddings of the line in the plane.} J. Reine Angew. Math. 276 (1975), 148--166.


\bibitem[CT]{CT}
Y. Chen, M. Tib\u ar, {\it Bifurcation values of mixed polynomials},  Math. Res. Lett. 19 (2012), no.1, 59--79.

\bibitem[CP]{CP}
M. Coste, M.J. de la Puente, \emph{ Atypical values at infinity of a polynomial function on the real plane: an erratum, and an algorithmic criterion.} J. Pure Appl. Algebra 162 (2001), no. 1, 23--35.

\bibitem[DD]{DD}
Z. Denkowska, M.P. Denkowski, \emph{The Kuratowski convergence and connected components},
 J. Math. Anal. Appl. 387 (2012) 48--65.

\bibitem[DRT]{DRT}
L.R.G. Dias, M.A.S. Ruas, M. Tib\u ar, \textit{Regularity at infinity of real mappings and a Morse-Sard theorem}, J. Topology, 5 (2012), no. 2, 323--340. 

\bibitem[HL]{HL}
 H\`a H.V.,   L\^e D.T., \emph{Sur la topologie des polyn\^omes complexes}, 
Acta Math. Vietnam. 9 (1984), no. 1, 21--32.


\bibitem[HN]{HN}
 H\`a H.V., Nguyen T.T.,  \emph{Atypical values at infinity of polynomial and rational functions on an algebraic surface in $\bR^n$.} Acta Math. Vietnam. 36 (2011), no. 2, 537-553.







\bibitem[KOS]{KOS}
 K. Kurdyka, P. Orro, S. Simon, {\em Semialgebraic Sard theorem for generalized critical values,} J. Differential Geometry 56 (2000), 67-92.





\bibitem[Pi]{Pi}
 S. Pinchuk, \emph{A counterexample to the strong real Jacobian conjecture}, Math. Z. 217 (1994), 1--4.

\bibitem[Ra]{Ra}
P.J. Rabier, {\em Ehresmann's fibrations and Palais-Smale conditions for morphisms of
  Finsler manifolds}, Ann. of Math. 146 (1997), 647--691.

\bibitem[ST]{ST}
 D. Siersma,  M. Tib\u ar, {\em Singularities at infinity
and their vanishing cycles},  Duke Math. Journal 80 (3) (1995), 771--783.

 
\bibitem[Su]{Su} M. Suzuki,
Propri\'et\'es topologiques des polyn\^omes de deux variables complexes, et automorphismes alg\'ebriques de l'espace $\bC^2$. 
J. Math. Soc. Japan 26 (1974), 241--257.



\bibitem[Ti1]{Ti-reg}
M. Tib\u ar, {\it Regularity at infinity of real and complex polynomial maps},  
  Singularity Theory, The C.T.C Wall Anniversary Volume,
 LMS Lecture Notes Series 263 (1999), 249--264. Cambridge 
University Press.

\bibitem[Ti2]{Ti}
M. Tib\u ar, Polynomials and vanishing cycles, Cambridge Tracts in Mathematics, 170, 
 Cambridge University Press 2007.

\bibitem[TZ]{TZ} 
M. Tib\u ar,  A. Zaharia, {\it Asymptotic behavior of families of real curves}. Manuscripta Math. 99 (1999), no. 3, 383--393.



\end{thebibliography}
\end{document}